\documentclass[12pt,reqno]{amsart}
\usepackage{amssymb,amscd}
\usepackage[alphabetic]{amsrefs}
\usepackage[all]{xy}
\usepackage[headings]{fullpage}
\usepackage[T1]{fontenc}
\usepackage{libertine}
\usepackage{stackengine}
\usepackage{mathtools}
\usepackage{graphicx}
\usepackage{comment}

\numberwithin{equation}{section}

\swapnumbers
\theoremstyle{plain}
\newtheorem{thm}[subsection]{Theorem}

\newtheorem{prop}[subsection]{Proposition}

\newtheorem*{thm*}{Theorem}
\newtheorem{prop-def}[subsubsection]{Proposition-Definition}

\theoremstyle{definition}

\newtheorem*{defn*}{Definition}

\theoremstyle{remark}
\newtheorem{rem}[subsection]{Remark}

\newtheorem{s-example}[subsection]{Example}

\numberwithin{equation}{subsection}

\usepackage[OT2,T1]{fontenc}
\DeclareSymbolFont{cyrletters}{OT2}{wncyr}{m}{n}
\DeclareMathSymbol{\sha}{\mathalpha}{cyrletters}{"58}

\renewcommand{\AA}{\mathcal{A}}

\newcommand{\CC}{\mathcal{C}}

\newcommand{\WW}{{\mathcal{W}}}

\newcommand{\OO}{\mathcal{O}}

\newcommand{\Fp}{{\mathbb{F}_p}}

\newcommand{\Z}{\mathbb{Z}}

\newcommand{\<}{\langle}
\renewcommand{\>}{\rangle}
\newcommand{\into}{\hookrightarrow}

\newcommand{\isoto}{\,\tilde{\to}\,}

\newcommand{\nodiv}{\not|}

\def\nodiv{\mathrel{\mathchoice{\not|}{\not|}{\kern-.2em\not\kern.2em|}
{\kern-.2em\not\kern.2em|}}}

\newcommand{\G}{\mathbb{G}}

\newcommand{\labeledto}[1]{\overset{#1}{\to}}

\DeclareMathOperator{\im}{Im}
\DeclareMathOperator{\Ker}{Ker}

\DeclareMathOperator{\Hom}{Hom}

\def\clap#1{\hbox to 0pt{\hss#1\hss}}

\makeatletter
\newcommand*\bigcdot{\mathpalette\bigcdot@{.5}}
\newcommand*\bigcdot@[2]{\mathbin{\vcenter{
               \hbox{\scalebox{#2}{$\m@th#1\bullet$}}}}}
\makeatother

\newcommand{\D}{\mathbb{D}}
\renewcommand{\and}{\quad\text{and}\quad}

\begin{document}
\title{Every $BT_1$ group scheme appears
in a Jacobian}

\author{Rachel Pries}
\address{Department of Mathematics \\ Colorado State University
  \\ Fort Collins, CO~~80523 USA}
\email{pries@math.colostate.edu}
\author{Douglas Ulmer}
\address{Department of Mathematics \\ University of Arizona
 \\ Tucson, AZ~~85721 USA}
\email{ulmer@math.arizona.edu}

\date{\today}

\subjclass[2010]{Primary 11D41, 11G20, 14F40, 14H40, 14L15;
Secondary 11G10, 14G17, 14K15, 14H10}

\keywords{Curve, finite field, Jacobian, abelian variety, Fermat
  curve, Frobenius, Verschiebung, group scheme,
   de Rham cohomology, Dieudonn\'e module}

\begin{abstract}
  Let $p$ be a prime number and let $k$ be an algebraically closed
  field of characteristic $p$.  A $BT_1$ group scheme over $k$ is a
  finite commutative group scheme which arises as the kernel of $p$ on
  a $p$-divisible (Barsotti--Tate) group.  Our main result is that
  every $BT_1$ scheme group over $k$ occurs as a direct factor of the
  $p$-torsion group scheme of the Jacobian of an explicit curve
  defined over $\Fp$.  We also treat a variant with polarizations.
  Our main tools are the Kraft classification of $BT_1$ group schemes,
  a theorem of Oda, and a combinatorial description of the de Rham
  cohomology of Fermat curves.
  \end{abstract}




\maketitle

\section{Introduction}
Fix a prime number $p$ and let $k$ be an algebraically closed field of
characteristic $p$.  A \emph{$BT_1$ group scheme} over $k$ is a finite
commutative group scheme which is the kernel of $p$ on a $p$-divisible
group.  (The term $BT_1$ stands for Barsotti--Tate truncated at level
1, and Barsotti--Tate is a synonym for $p$-divisible.)  These are the
finite commutative group schemes killed by $p$ which also satisfy
$\Ker F=\im V$ and $\Ker V=\im F$ where $F$ and $V$ are the Frobenius
and Verschiebung maps respectively.  The simplest $BT_1$ group
schemes are $\Z/p\Z$ and $\mu_p$.

We also consider \emph{polarized $BT_1$ group schemes} over $k$,
i.e., $BT_1$ group schemes $G$ with a pairing that induces a
non-degenerate, alternating pairing on the Dieudonn\'e module of $G$,
as in \cite[\S9]{Oort01}.  If $A$ is a principally polarized abelian
variety of dimension $g$ over $k$ , its $p$-torsion subscheme $A[p]$
is naturally a polarized $BT_1$ group scheme of order $p^{2g}$.

If $C$ is a smooth irreducible projective curve of genus $g$ over
$k$, then its Jacobian ${\rm Jac}(C)$ is a principally polarized
abelian variety of dimension $g$, and thus $G={\rm Jac}(C)[p]$ is a
polarized $BT_1$ group scheme of rank $p^{2g}$.  By a result of Oda
\cite{Oda69}, the de Rham cohomology of $C$ over $k$ determines the
isomorphism class of $G$ uniquely via its Dieudonn\'e module.

In general, it is not known which polarized $BT_1$ group schemes occur
for Jacobians of curves.  In fact, there are very few examples of
curves for which the isomorphism class of ${\rm Jac}(C)[p]$ has been
computed.  Our main result is:

\begin{thm}\label{thm:main}\mbox{}
\begin{enumerate}
\item Every $BT_1$ group scheme over $k$ appears as a direct factor of
  ${\rm Jac}(C)[p]$ for an explicit curve $C$ defined over $\Fp$.
\item Every polarized $BT_1$ group scheme over $k$ appears as a direct
  factor \textup{(}with pairing\textup{)} of ${\rm Jac}(C)[p]$
  for an explicit curve $C$ defined over $\Fp$.
\item In particular, if $G$ is an indecomposable $BT_1$ group scheme
  of rank $p^{\ell}$ with $\ell>1$, or if $G$ is an indecomposable
  polarized $BT_1$ group scheme of rank $p^{\ell}$ with $\ell>2$, then
  the curve $C$ in part \textup{(1)} or part \textup{(2)} can be
  chosen to have genus $\le(p^{\ell}-2)/2$.
\end{enumerate}
\end{thm}

We prove this theorem in Section~\ref{s:proof}.  
Parts (1) and (2) are essentially equivalent, but
neither implication is immediate.  In part (3), a polarized $BT_1$
group scheme is indecomposable if it is not the orthogonal direct sum
of two proper polarized subgroup schemes.  The restrictions on $\ell$
in (3) are not significant, because the omitted groups are known to
appear in elliptic curves.

A weaker version of parts (1) and (2) follows from the fact that every
abelian variety appears as a subvariety of a Jacobian together with
the non-emptiness of each E--O stratum of $\AA_g$; see
Remark~\ref{rem:otherway}.  Our proof of Theorem~\ref{thm:main} is
more elementary, and it yields a stronger result because: (i) there
are no conditions on $p$; (ii) the curve $C$ is explicit and its field
of definition is $\Fp$; (iii) the genus of $C$ is bounded in terms of
the rank of $G$; and (iv) the other group schemes that occur in
${\rm Jac}(C)[p]$ can be explicitly computed.

In almost all cases the ``explicit curve'' of the theorem can be taken
to be a quotient of a Fermat curve.  Fermat curves are a natural class
of curves to consider because their de Rham cohomology, with its
Frobenius and Verschiebung operators, admits a simple combinatorial
description.  A result of independent interest in this paper is that
we determine the structure of the $BT_1$ module for the Jacobian of
the Fermat curve $F_d$ of degree $d$, for all positive integers $d$
that are relatively prime to $p$, see Theorem~\ref{thm:BT1Fd}.
This complements work of Yui, who determined the Newton
polygons of Fermat curves \cite[Thm.~4.2]{Yui80}.

Our arguments use parts of three classifications of $BT_1$ group
schemes (in terms of words, the ``canonical filtration'', and
permutations), largely due to Kraft, Ekedahl, and Oort.  In a
companion paper \cite{PriesUlmerFEO}, we provide a complete
translation between these classifications, and we apply them to give a
detailed study of the $p$-torsion subgroups of Jacobians of Fermat
curves, including well-known invariants like the $p$-rank and
$a$-number, as well as two other invariants related to supersingular
elliptic curves.

\subsection*{Acknowledgements} 
Author RP was partially supported by NSF grant DMS-1901819, and
author DU was partially supported by Simons Foundation grants 359573
and 713699.

\section{Groups and modules}

\subsection{Dieudonn\'e modules}

We refer to \cite{Fontaine77} for background on contravariant
Dieudonn\'e theory for finite group schemes of $p$-power order over
$k$ and for $p$-divisible groups.

Write $\sigma$ for the absolute Frobenius of $k$ and extend it to
the Witt vectors $W(k)$.  Define the \emph{Dieudonn\'e ring}
$\D=W(k)\{F,V\}$ as the $W(k)$-algebra generated by $F$ and $V$ with
relations
\begin{equation*}
  FV=VF=p, \ F\alpha=\sigma(\alpha)F,
  \text{ and } \alpha V=V\sigma(\alpha) \text{ for } \alpha\in W(k).
\end{equation*}  Let
$\D_k=\D/p\D\cong k\{F,V\}$.

If $G$ is a finite, commutative group scheme over $k$ of $p$-power
order, then its \emph{Dieudonn\'e module} $M(G)$ is the left
$\D$-module of homomorphisms of $k$-group schemes from $G$ to the
co-Witt vectors.  The functor $G\leadsto M(G)$ is contravariant and
induces an anti-equivalence between the category of finite group
schemes of $p$-power order over $k$ and the category of left
$\D$-modules that are of finite length as $W(k)$ modules
\cite[III.1.4]{Fontaine77}.

\subsection{$BT_1$ group schemes and $BT_1$ modules} 
By definition, a \emph{$BT_1$ group scheme} over $k$ is a finite
commutative group scheme $G$ that is killed by $p$ and that has the
properties
\[\Ker(F:G\to G)=\im(V:G\to G)\quad\text{and}\quad
  \im(F:G\to G)=\Ker(V:G\to G).\]
The notation $BT_1$ is an abbreviation of ``Barsotti--Tate of level 1''.

By definition, a \emph{$BT_1$ module} over $k$ is a $\D_k$-module $M$
of finite dimension over $k$ such that
\[\Ker(F:M\to M)=\im(V:M\to M)\quad\text{and}\quad
  \im(F:M\to M)=\Ker(V:M\to M).\] A $\D_k$-module $M$ is a $BT_1$
module if and only if $M=M(G)$ for a $BT_1$ group scheme $G$ over $k$.

The group schemes $\Z/p\Z$ and $\mu_p$ are $BT_1$ group schemes.  So
is $G_{1,1}$, the kernel of $p$ on a supersingular elliptic curve over
$k$.  On the other hand, $\alpha_p$ is not.

\subsection{Duality}
If $G$ is a finite, commutative group scheme over $k$, define its
\emph{Cartier dual} $G^D$ as $G^D:=\Hom_{k-Gr}(G,\G_m)$,
(homomorphisms of $k$-group schemes)
where $\G_m$ is the multiplicative group over $k$.  
A $BT_1$ group scheme $G$ is
\emph{self-dual} if $G\cong G^D$.  

If $M$ is a left $\D$-module of finite length over $W(k)$, define its
\emph{dual module} $M^*$ as follows: If $M$ is killed by $p^n$, set
$M^*=\Hom_{W(k)}(M,W_n(k))$ with $(F\phi)(m)=\sigma(\phi(Vm))$ and
$(V\phi)(m)=\sigma^{-1}(\phi(Fm))$ for all $\phi\in M^*$ and $m\in M$.
A $BT_1$ module $M$ is \emph{self-dual} if $M\cong M^*$.

A basic result of Dieudonn\'e theory
\cite[\S III.5.3]{Fontaine77} is that $M(G^D)\cong M(G)^*$.  In
particular, $G$ is self-dual if and only if $M(G)$ is self-dual.

\subsection{Polarized $BT_1$ group schemes}  
A \emph{polarized $BT_1$ module} is a $BT_1$ module $M$ equipped with
a non-degenerate, alternating pairing
$\<\cdot, \cdot\>:M\times M\to k$ of Dieudonn\'e modules (i.e., such
that $\<x,x\>=0$ and $\<Fx,y\>=\<x,Vy\>^p$ for all $x,y\in M$).
Clearly, a polarized $BT_1$ module is self-dual.

A \emph{polarized $BT_1$ group scheme} is a $BT_1$ group scheme $G$
equipped with a bilinear form with the property that the induced form
on $M(G)$ is non-degenerate and alternating.  (The reason for this
unusual definition is that when $p=2$, an alternating form on $G$ need
not induce an alternating form on $M(G)$.  See
\cite[p.~346]{Oort01}.)

Oort proved \cite[\S\S2, 5, 9]{Oort01} (see also
\cite[Cor.~4.2.3]{PriesUlmerFEO}) that any self-dual $BT_1$ module can
be given a unique polarization:

\begin{prop}\label{prop:self-dual=>polarized}
  Every self-dual $BT_1$ module admits a polarization, i.e., a
  non-degenerate \emph{alternating} pairing, and this pairing is
  unique up to \textup{(}non-unique\textup{)} isomorphism.
\end{prop}

\section{The Kraft classification of $BT_1$ modules} 
In this section, we review a bijection due to Kraft between
isomorphism classes of $BT_1$ modules over $k$ and certain data
obtained from words on a two-letter alphabet, see \cite{Kraft75}.

\subsection{Words}\label{ss:words}
Let $\WW$ be the monoid of words $w$ on the two-letter alphabet
$\{f,v\}$, and write $1$ for the empty word.  The \emph{complement}
$w^c$ of $w$ is the word obtained by exchanging $f$ and $v$ at every
letter.

For a positive integer $\lambda$, write $\WW_\lambda$ for the words of
length $\lambda$.  If $w \in \WW_\lambda$, we write
$w=u_{\lambda-1}\cdots u_0$ where $u_i\in\{f,v\}$ for
$0 \leq i \leq \lambda-1$.  Define an action of the group $\Z$ on
$\WW$ by requiring that $1\in\Z$ map $w=u_{\lambda-1}\cdots u_0$ to
$u_0u_{\lambda-1}\cdots u_1$.  If $w$ and $w'$ are in the same orbit
of this action, we say $w'$ is a \emph{rotation} of $w$.  The orbit
$\overline{w}$ of $w$ under the action of $\Z$ is called a
\emph{cyclic word}.  Write $\overline\WW$ for the set of cyclic words.

A word is \emph{primitive} if it is not a power of a shorter word,
i.e., not of the form $w^e$ for some integer $e>1$.  Write $\WW'$ for
the set of primitive words.

The three ``decorations'' of $\WW$ (length, cyclic, primitive) may be
applied in any combination.  Thus $\overline\WW'_\lambda$ stands
for the set of primitive cyclic words of length $\lambda$.

\subsection{Cyclic words to $BT_1$ modules}\label{ss:words-to-BT1s}
Following Kraft \cite{Kraft75}, we attach a $BT_1$ module to a
multiset of primitive cyclic words.

Suppose that $w\in\WW_\lambda$ is a word of length $\lambda$, say
$w=u_{\lambda-1}\cdots u_0$ with $u_j\in\{f,v\}$.  Let $M(w)$ be the
$k$-vector space with basis $e_j$ with $j\in\Z/\lambda\Z$ and define a
$p$-linear map $F:M(w)\to M(w)$ and a $p^{-1}$-linear map
$V:M(w)\to M(w)$ by setting
\[F(e_{j})=\begin{cases}
    e_{{j+1}}&\text{if $u_j=f$},\\
    0&\text{if $u_j=v$,}
  \end{cases}
\qquad\text{and}\qquad V(e_{{j+1}})=
  \begin{cases}
    e_{{j}}&\text{if $u_j=v$},\\
    0&\text{if $u_j=f$.}  \end{cases}\]
This construction yields a $BT_1$ module of dimension $\lambda$ over
$k$ which up to isomorphism only depends on the cyclic word
$\overline{w}$ associated to $w$.

Kraft proves that if $w$ is primitive then $M(w)$ is indecomposable,
and that every indecomposable $BT_1$ module is isomorphic to one of
the form $M(w)$ for a unique primitive cyclic word $\overline{w}$.
Thus every $BT_1$ module $M$ is isomorphic to a direct sum
$\oplus M(w_i)$ where $\overline{w}_i$ runs through a uniquely
determined multiset of primitive cyclic words.

Even if $w$ is not primitive, the formulas above define a $BT_1$
module.  If $w=(w')^e$, Kraft also proves that $M(w)\cong M(w')^e$.

It is clear that $M(f)=M(\Z/p\Z)$, $M(v)=M(\mu_p)$, and $M(fv)$ is the
Dieudonn\'e module of the kernel of $p$ on a supersingular elliptic
curve.  More generally, if $w$ has length $>1$ and is primitive, then
$M(w)$ is the Dieudonn\'e module of a unipotent, connected $BT_1$
group scheme.

\subsection{Duality}
It is clear from the definitions that duality of modules corresponds
to complementation of words, i.e., $M(w)^*\cong M(w^c)$.  It follows
that an indecomposable, self-dual $BT_1$ module is either of the form
$M(w)$ where $w$ induces a self-complementary cyclic word
($\overline w^c=\overline w$) or of the form $M(w)\oplus M(w^c)$ where
$\overline w^c\neq\overline w$.

\section{Permutations and $BT_1$ modules}
In this section, we associate a $BT_1$ module to certain permutations
via the Kraft construction.

\subsection{Permutations}
Consider a finite set $S$ written as the disjoint union
$S=S_f\cup S_v$ of two subsets. 
Let $\pi:S\to S$ be a permutation of $S$.  Two such collections of
data $(S=S_f\cup S_v,\pi)$ and $(S'=S'_f\cup S'_v,\pi')$ are
isomorphic if there is a bijection $\iota: S\to S'$
such that $\iota(S_f)=S_f'$, $\iota(S_v)=S_v'$, 
and $\iota \pi = \pi' \iota$.

\subsection{Permutations to words}\label{ss:perms->words}
Given $S=S_f\cup S_v$ and $\pi$ as above, we define a multiset of
cyclic words as follows:   For $a\in S$ with orbit of
size $\lambda$, define the word $w_a=u_{\lambda-1}\cdots u_0$ where
\[u_j=\begin{cases}
  f&\text{if $\pi^j(a)\in S_f$},\\
  v&\text{if $\pi^j(a)\in S_v$}.
\end{cases}\]
Then $\overline{w}_a$ depends only on the orbit of $a$.  This gives a
well-defined map from orbits of $\pi$ to cyclic words, i.e., elements
of $\overline\WW$.  Taking the union over orbits, we can associate to
$(S=S_f\cup S_v,\pi)$ a multiset of cyclic words.  If $S$ and $S'$ are
isomorphic, then they yield the same multiset.

For example, let $S=\{1,\dots,9\}$, $S_f=\{2,3,5,6,9\}$, and
$S_v=\{1,4,7,8\}$.  Let $\pi$ be the permutation $(135)(246)(789)$.
The orbit through 1 and the orbit through 2 both give rise to the
cyclic word $\overline{ffv}$, and the orbit through 7 gives rise to
the cyclic word $\overline{fvv}$.  The associated multiset is
$\{(\overline{ffv})^2,\overline{fvv}\}$ (where $(\overline{ffv})^2$
means the cyclic word $\overline{ffv}$ taken with multiplicity 2).

\subsection{Permutations to $BT_1$ modules}\label{ss:perms->modules}
Given $S=S_f\cup S_v$ and $\pi$ as above, we obtain a multiset of
words, and thus a $BT_1$ module of dimension equal to the cardinality
of $S$.  This $BT_1$ module can be
described directly in terms of $S$ as follows:  Form the $k$-vector
space $M(S)$ with basis elements $\{e_a \mid a\in S\}$ and define
a $p$-linear map $F:M(S)\to M(S)$ and a $p^{-1}$-linear map
$V:M(S)\to M(S)$ by setting
\[F(e_{a})=\begin{cases}
    e_{\pi(a)}&\text{if $a\in S_f$},\\
    0&\text{if $a\in S_v$,}
  \end{cases}
\qquad\text{and}\qquad V(e_{\pi(a)})=
  \begin{cases}
    e_{a}&\text{if $a\in S_v$},\\
    0&\text{if $a\in S_f$.}  \end{cases}\]
Note that $M(S)$ decomposes (as a $BT_1$ module) into submodules indexed
by the orbits of $\pi$.  The submodule corresponding to an orbit of
$\pi$ is indecomposable if and only if the word associated to the
orbit is primitive.

\subsection{Duality}
Given $S=S_f\cup S_v$ and $\pi$ as above, form the $BT_1$ module
$M(S)$.  It is clear from the definitions that the dual module
$M(S)^*$ is the module associated to data
$(S^*=S^*_f\cup S^*_v,\pi^*)$ where $S^*=S$, $S^*_f=S_v$, $S^*_v=S_f$,
and $\pi^*=\pi$.  It follows that $M$(S) is self-dual if and only if
there exists a bijection $\iota:S\isoto S$ which satisfies
$\iota(S_f)=S_v$ and $\pi\circ \iota=\iota\circ \pi$.

\section{Fermat Jacobians}\label{s:FJ}
In this section, 
we study $p$-torsion group schemes of Jacobians of
Fermat curves.

\subsection{$BT_1$ modules associated to curves}
Let $\CC$ be an irreducible, smooth, projective curve of genus $g$
over $k$, and let $J=J_\CC$ be its Jacobian.  In \cite[\S5]{Oda69},
Oda gives $H^1_{dR}(\CC)=H^1_{dR}(J)$ the structure of a $BT_1$
module.
In particular, writing $H$ for $H^1_{dR}(\CC)$, we have
\begin{equation}\label{eq:ImV=KerF}
\im(V:H\to H)=\Ker(F:H\to H)\cong H^0(\CC,\Omega^1_\CC),  
\end{equation}
and
\[\im(F:H\to H)=\Ker(V:H\to H)\cong H^1(\CC,\OO_\CC).\]
Oda proves \cite[Cor.~5.11]{Oda69} that there is a
canonical isomorphism of $\D_k$-modules
\begin{equation}\label{eq:Oda}
H^1_{dR}(\CC)\cong M(J[p]).  
\end{equation}

\subsection{Fermat curves}
For each positive integer $d$ not divisible by $p$, let $F_d$ be the
Fermat curve of degree $d$, i.e., the smooth, projective curve over $k$
with affine model $F_d: X^d+Y^d=1$ and let $J_{F_d}$ be its Jacobian.
Let $\CC_d$ be the smooth, projective curve over $k$ with affine model 
\begin{equation}\label{eq:Cd-model}
\CC_d:\qquad y^d=x(1-x), 
\end{equation}
and let $J_d$ be its Jacobian.

The curve $\CC_d$ is a quotient of $F_d$.  (Substitute $X^d$ for $x$
and $XY$ for $y$ in the equation for $\CC_d$.)  The map
$F_d \to \CC_d$ is the quotient of $F_d$ by a subgroup of
$(\mu_d)^2 \subset {\rm Aut}(F_d)$ of index $d$.
Since the degree of $F_d\to\CC_d$ is prime to $p$, $J_d[p]$ is a
direct factor of $J_{F_d}[p]$.  Almost all of the interesting features
of $J_{F_d}[p]$ are already present in $J_d[p]$, so for simplicity we
use $\CC_d$ whenever possible.

\subsection{Cohomology of $\CC_d$}
The Riemann-Hurwitz formula shows
that the genus of $\CC_d$ is
\[g(\CC_d)=\lfloor (d-1)/2\rfloor=
  \begin{cases}
  (d-1)/2&\text{if $d$ is odd,}\\
    (d-2)/2&\text{if $d$ is even.}
  \end{cases}\]
Moreover, $\CC_d$ admits an action of $\zeta \in \mu_d$
with $\zeta:(x,y)\mapsto(x,\zeta y)$. 

We next describe $H^1_{dR}(\CC_d)$ in a form conducive to
studying it as a $\D_k$-module.  First, write
\[H^1_{dR}(\CC_d)=\bigoplus_{a\in\Z/d\Z}H_a,\]
where $H_a$ is the subspace of $H^1_{dR}(\CC_d)$ where every
$\zeta\in\mu_d$ acts by multiplication by $\zeta^a$.  Since the action
of $\mu_d$ on $\CC_d$ induces the trivial action on $H^2_{dR}(\CC_d)$, 
the cup product induces a perfect duality between $H_a$ and
$H_{-a}$, and a trivial pairing between $H_a$ and $H_b$ if
$b\not\equiv -a\pmod d$.

Let
\[S=
  \begin{cases}
    \Z/d\Z\setminus\{0\}&\text{if $d$ is odd,}\\
    \Z/d\Z\setminus\{0,d/2\}&\text{if $d$ is even.}
  \end{cases}\]
Multiplication by $p$ induces a permutation of $S$.  
We make sense of any archimedean statement about an
element $a\in S$ (e.g., ``$0<a<d/2$'') by
implicitly lifting $a$ to its least positive residue.

\begin{prop}\label{prop:cohom}
\mbox{}
  \begin{enumerate}
  \item If $a\in \Z/d\Z$, then $\dim_k (H_a)=1$ if $a\in S$ and
    $H_a=0$ if $a\not\in S$.
  \item $H^0(\CC_d,\Omega^1_{\CC_d})=\oplus_{0<a<d/2}H_a$.
    \item If $0<a<d/2$, then $FH_a=0$ and $V$ induces an isomorphism
      $V:H_{pa}\to H_a$.
      \item If $d/2<a<d$, then $VH_a=0$ and $F$ induces an isomorphism
        $F:H_a\to H_{pa}$.
  \end{enumerate}
\end{prop}

There are several similar calculations in the literature (e.g.,
\cite{Weil76}, \cite[\S5]{Dummigan95}, and \cite[\S6]{Ulmer14d}).  For
the convenience of the reader, we include the following efficient and
transparent proof.

\begin{proof}
  For $0<a<d/2$, a simple calculation shows that the 1-form
  $y^adx/y^d$ on the affine model \eqref{eq:Cd-model} extends to a
  global 1-form on $\CC_d$, and its class in $H^1_{dR}(\CC_d)$
  lies in $H_a$.  This shows that $\dim_k (H_a)\ge1$ for $0<a<d/2$. 
  Because of the perfect duality between $H_a$ and $H_{-a}$, 
  we see that $\dim_k(H_a)\ge1$ for $d/2<a<d$.  Since $g(\CC_d) = 
  \lfloor(d-1)/2\rfloor$, it follows that $\dim_k(H_a)=1$ for
  $a\in S$ and $H_a=0$ for $a\not\in S$.  This proves parts (1) and (2).

  By definition, $FH_a\subset H_{pa}$ and
  $VH_{pa}\subset H_a$. 
  By \eqref{eq:ImV=KerF}, $F$ kills
 $H^0(\CC_d,\Omega^1_{\CC_d})=\oplus_{0<a<d/2}H_a$.
  Since $\dim(\Ker F)=g$,
  the map $F:H_a\to H_{pa}$ is injective, and thus bijective, for $d/2<a<d$.
  Similarly, by \eqref{eq:ImV=KerF},
  $\im V=H^0(\CC_d,\Omega_{\CC_d}^1)$.  So $V:H_{pa}\to H_a$ is
  surjective, and thus bijective, for $0<a<d/2$ and zero for
  $d/2<a<d$.  This proves parts (3) and (4).
\end{proof}

Now let $S_f,S_v\subset S$ be given by
\[S_f=\{a \mid d/2<a<d\}\quad\text{and}\quad S_v=\{a \mid 0<a<d/2\},\]
and let $\pi:S\to S$ be the permutation induced by multiplication by
$p$.

\begin{thm} \label{thm:BT1fermat}
  The Dieudonn\'e module $M(J_d[p])$ is the $BT_1$ module associated to
  the data
  \begin{equation*}
  S=
       \begin{cases}
\Z/d\Z\setminus\{0,d/2\}&\text{if $d$ is even}\\         
\Z/d\Z\setminus\{0\}&\text{if $d$ is odd},
\end{cases}
\end{equation*}
\begin{equation*}
  S_f =\left\{a\in S \mid d/2<a<d \right\}, \quad
  S_v=\left\{a\in S \mid 0<a<d/2 \right\},
\end{equation*}
and the permutation $\pi:S\to S$ given by $\pi(i)=pa$.
\end{thm}

\begin{proof}
  This is immediate from Proposition~\ref{prop:cohom} and
  Section~\ref{ss:perms->modules}. 
\end{proof}

\begin{rem}
  The proof of the theorem shows that $J_d[p]$ decomposes as a direct
  sum over the orbits of $\pi$ on $S$.  By
  Sections~\ref{ss:words-to-BT1s} and \ref{ss:perms->modules}, the
  summand corresponding to an orbit is indecomposable if and only if
  the word associated to the orbit is primitive.
\end{rem}

\begin{rem}
The data
  $(S=S_f\cup S_v,\pi)$ also completely determines $J_d[p]$ up to
  isomorphism as a polarized $BT_1$ group scheme.  Indeed, by
  Proposition~\ref{prop:self-dual=>polarized}, given any
  non-degenerate, alternating form on the $BT_1$ module defined by
  $(S=S_f\cup S_v,\pi)$, we may choose the isomorphism in the theorem
  so that it intertwines the given form with the polarization on
  $H^1_{dR}(\CC_d)$ induced by the cup product. 
 \end{rem}

\subsection{The $p$-torsion of Fermat curves}\label{ss:Fd}

In this section, we determine the $BT_1$ modules of the Jacobians of
Fermat curves.  We need this material to complete
Theorem~\ref{thm:main} when $p<5$.

First, note that $\mu_d^2$ acts on $F_d$ via
$(\zeta_1,\zeta_2):(X,Y)\mapsto(\zeta_1X,\zeta_2Y)$.  The cohomology
$H=H^1_{dR}(F_d)$ decomposes into subspaces $H_{a,b}$ indexed by
$(a,b)\in(\Z/d\Z)^2$ on which $(\zeta_1,\zeta_2)\in\mu_d^2$ acts by
$\zeta_1^a\zeta_2^b$.  An argument parallel to
Proposition~\ref{prop:cohom} shows that $H_{a,b}$ is 1-dimensional if
$(a,b)\in T$ and $0$ otherwise, where
\begin{equation*}
T =\left\{(a,b)\in(\Z/d\Z)^2 \mid a\neq0,b\neq0,a+b\neq0 \right\}.
\end{equation*}
Moreover, setting
\begin{equation*}
 T_f =\left\{(a,b)\in S \mid a+b>d\right\},
\text{ and }T_v=\left\{(a,b)\in S \mid a+b<d\right\},
\end{equation*}
then $F$ induces an isomorphism $F:H_{a,b}\to H_{pa,pb}$ if
$(a,b)\in T_f$, and $V$ induces an isomorphism
$V:H_{pa,pb}\to H_{a,b}$ if $(a,b)\in T_v$.  Consider the permutation
$\sigma:T\to T$ given by $\sigma(a,b)=(pa,pb)$.  We may associate
words to elements of $T$ and cyclic words to orbits of $\sigma$.  As
in Section~\ref{ss:perms->modules}, this defines a $BT_1$ module.
Applying Oda's Theorem \eqref{eq:Oda} proves that:
\begin{thm}\label{thm:BT1Fd}
The module $M(J_{F_d}[p])$ is the $BT_1$
module associated to the data $(T=T_f\cup T_v,\sigma)$.  
\end{thm}

\subsection{A Shimura variety perspective}
Another proof of Theorem~\ref{thm:BT1fermat} can be extracted from
\cite{Moonen04} as follows.  By \cite{Weil76}, $J_d$ is an abelian
variety of CM type and Proposition~\ref{prop:cohom}(1-2) reveals the
CM type of $J_d$.  This yields a zero-dimensional Shimura subvariety
of $\AA_g$; each of its points lies in the open stratum of the
stratification by either Newton polygon or Ekedahl--Oort type.  By
\cite[\S1]{Moonen04}, one can compute the isomorphism type of $J_d[p]$
in terms of the CM type.

\subsection{Other related work}
The curve $\CC_d$ is hyperelliptic.  When $p=2$,
Theorem~\ref{thm:BT1fermat} is a special case of \cite{ElkinPries13},
where the authors compute the $BT_1$ module and E--O type for all
hyperelliptic curves in characteristic 2.  When $p$ is odd,
Devalapurkar and Halliday compute the action of $F$ and $V$ on the mod
$p$ Dieudonn\'e module for every hyperelliptic curve
\cite{DevalapurkarHallidaypp17}.  However, it appears to be difficult
to deduce Theorem~\ref{thm:BT1fermat} from their result because it
describes the actions of $F$ and $V$ by unwieldy formulas.  Our work
gives them essentially as permutation matrices.

In \cite{Moonenpp20}, the author gives a method for computing the
$BT_1$ module of a smooth complete intersection curve over a field
whose characteristic is greater than the largest of the multidegrees.
This method can be used to recover a version of
Theorem~\ref{thm:BT1fermat} for the Fermat curve $F_d$ in
characteristic $p$ when $d<p$.  However, this is not adequate to prove
Theorem~\ref{thm:main} because most $BT_1$ modules do not appear in
Jacobians of Fermat curves of degrees $d<p$.

\section{Proof of Theorem~\ref{thm:main}}\label{s:proof}
We will prove most cases of Theorem~\ref{thm:main} by considering the
Fermat quotient curve $\CC_d$ for $d$ of the form $p^\ell-1$.  When
$p< 5$, we also need the Fermat curve $F_d$ and an auxiliary fiber
product.

\subsection{$p$-adic digits}
Fix a positive integer $\ell$ and let $d=p^\ell-1$.   As before, let
  \begin{equation*}
  S=
       \begin{cases}
\Z/d\Z\setminus\{0,d/2\}&\text{if $d$ is even}\\         
\Z/d\Z\setminus\{0\}&\text{if $d$ is odd},
\end{cases}
\end{equation*}
\begin{equation*}
  S_f =\left\{a\in S \mid d/2<a<d \right\}, \quad
  S_v=\left\{a\in S \mid 0<a<d/2 \right\},
\end{equation*}
and let the permutation $\pi:S\to S$ be given by $\pi(a)=pa$.  Given
$a\in S$, consider the $p$-adic expansion of its least positive residue:
\[a=a_0+a_1p+\cdots+a_{\ell-1}p^{\ell-1},\]
where $a_i\in\{0,\dots,p-1\}$.  We exclude: $a=0$ (all $a_i=0$);
$a=d$ (all $a_i=p-1$); and, if $p$ is odd, $a=d/2$ (all $a_i=(p-1)/2$).
Note that
\[pa\equiv a_{\ell-1}+a_0p+\cdots+a_{\ell-2}p^{\ell-1}\pmod d,\]
so $\pi$ corresponds to permuting the digits of $a$ cyclically.

By definition, $a\in S_v$ if and only if $0<a<d/2$.  In terms of
digits, this holds if and only if 
\begin{align*}
  a_{\ell-1}&<(p-1)/2, \text{ or }\\
    a_{\ell-1}&=(p-1)/2\quad\text{and}\quad a_{\ell-2}<(p-1)/2, \text{ or }\\
    a_{\ell-1}&=a_{\ell-2}=(p-1)/2\quad\text{and}\quad
                a_{\ell-3}<(p-1)/2,  \text{ or } \ldots
\end{align*}
In other words, the condition is that the first $p$-adic digit to the left of
$a_{\ell-1}$ (inclusive) which is not $(p-1)/2$ is in fact
less than $(p-1)/2$.  

Similarly $a\in S_f$ if and only if $a>d/2$.  This holds if and only
if the first $p$-adic digit to the left of $a_{\ell-1}$ (inclusive)
which is not $(p-1)/2$ is in fact greater than $(p-1)/2$.

For $a\in S$, let $\lambda$ be the size of the orbit of $\pi$
through $a$.  We say that $a$ is \emph{primitive} if $\lambda=\ell$.
It is clear that $a$ is primitive if $gcd(d,a)=1$.  More generally,
$a$ fails to be primitive if and only if $d/\gcd(d,a)$ divides
$p^\lambda-1$ for some $\lambda<\ell$.  Also, $a$ is primitive if
$w_a$ is primitive (in the sense of Section~\ref{ss:words}), but not
conversely in general.

Let $w_a=u_{\lambda-1}\cdots u_0$ be the word attached to $a$ as in
Section~\ref{ss:perms->words}.  The discussion above shows that
$u_j=v$ if and only if the first $p$-adic digit of $a$ to the left of
$a_{\ell-1-j}$ (inclusive) which is not $(p-1)/2$ is in fact less than
$(p-1)/2$.  (Finding the first such digit may require ``wrapping
around,'' i.e., passing from $a_0$ to $a_{\ell-1}$.)

Using these observations, we may write down elements $a\in S$ with
given words:

\begin{prop}\label{prop:mult>0}
  \mbox{}
  Suppose $w$ is a word of length $\ell > 1$. 
  \begin{enumerate}
  \item If $w$ is primitive, then there is
    an element $a\in S$ such that $w_a=w$.
  \item If $p>3$ and $w$ is any word \textup{(}not necessarily
    primitive\textup{)}, then there is an element $a\in S$ such that
    $w_a=w$.
  \end{enumerate}
\end{prop}

\begin{proof}
  (1) Let $w=u_{\ell-1}\cdots u_0$ be a primitive word of length
  $\ell>1$.  For $0\le j<\ell$, set
  \[a_j=\begin{cases}
      0&\text{if $u_{\ell-1-j}=v$},\\
      p-1&\text{if $u_{\ell-1-j}=f$}.
    \end{cases}\]
  Since $w$ is primitive, and in particular not equal to $f^\ell$ nor
  to $v^\ell$, the integer $a=a_0+\cdots+a_{\ell-1}p^{\ell-1}$ defines
  an element of $S$, and it is clear that $w_a=w$.

  (2) If $w=v^\ell$ (resp. $w=f^\ell$), we may take $a=1$
  (resp. $a=d-1$).  (Here we use $p>2$.)  For any other word, the
  recipe in the preceding paragraph yields an element of $S$.
  However, if $w$ is not primitive, say $w={w'}^e$, this element is not
  what we need because its word is $w'$.  Modify $a$ as follows:
  choose $j$ so that $a_j=0$ (which exists because $w\neq v^\ell$),
  and change $a_j$ to 1.  Then the new $a$ is primitive (because
  exactly one of its digits is 1) and satisfies $w_a=w$.  (Here we use
  that $1<(p-1)/2$, i.e., that $p>3$.)  This
  completes the proof of the proposition.
  \end{proof}

  \begin{rem}
    In \cite{PriesUlmerFEO}, we give a more refined analysis and
    compute the number of $a\in S$ with $w_a=w$ for any $w$, $p$, and
    $\ell$.  It turns out that the restriction on $p$ in part (2) is
    essential.  If $p=3$ and $e>1$, then the word $(fv)^e$ is not the
    word associated to an element of $S$, and if $p=2$, $e>1$, and
    $w'$ is non-trivial, then ${w'}^e$ is also not associated to an
    element of $S$.
  \end{rem}

  \subsection {Proof of Theorem~\ref{thm:main}, part (3)}
  Let $G$ be an indecomposable $BT_1$ group scheme over $k$ of order
  $p^\ell$ with $\ell>1$.  Then there is a primitive word $w$ of length
  $\ell$ such that $M(G)\cong M(w)$.  According to
  Proposition~\ref{prop:mult>0}, there is an element $a\in S$ such
  that $w_a=w$.  By Theorem~\ref{thm:BT1fermat}, $G$ appears as a
  direct factor of $J_d[p]$.  Since $\CC_d$ has genus
  $\lfloor(p^\ell-2)/2\rfloor$, this establishes the desired result
  for an indecomposable $BT_1$.

  Now consider an indecomposable polarized $BT_1$ group scheme $G$.
  If $G$ is indecomposable as a $BT_1$ group scheme (ignoring the
  pairing), the proof in the previous paragraph applies with the same
  bound on the genus.  Otherwise, there is a word $w$ of length
  $\ell/2$ such that $M(G)\cong M(w)\oplus M(w^c)$.  Let
  $d'=p^{\ell/2}-1$ and let $S'$ be the usual set for $d'$:
\[S'=
  \begin{cases}
    \Z/d'\Z\setminus\{0\}&\text{if $d'$ is odd,}\\
    \Z/d'\Z\setminus\{0,d/2\}&\text{if $d'$ is even.}
  \end{cases}\]
Since $\ell/2>1$, by Proposition~\ref{prop:mult>0}, there is an
element $a\in S'$ such that $w_a=w$ (and so $w_{-a}=w^c$).  By
Theorem~\ref{thm:BT1fermat}, $G$ is a direct factor of
$J_{d'}[p]$.  To confirm the bound on the genus, we note that 
$\CC_{d'}$ has genus $\lfloor(p^{\ell/2}-2)/2\rfloor$.

In the polarized case, if $G$ is a direct factor of $J_d[p]$, we check
that the given pairing on $G$ is induced from that of $J_d[p]$ using
Oort's result (Proposition~\ref{prop:self-dual=>polarized}).  The same
argument applies if $G$ is a direct factor of $J_{d'}[p]$.  Write
$J_d[p]\cong G\oplus G'$.  Since $J_d[p]$ and $G$ are self-dual, so is
$G'$.  By the existence part of
Proposition~\ref{prop:self-dual=>polarized}, $G$ and $G'$ both admit
polarizations, and by the uniqueness part, we may choose the
isomorphism $J_d[p]\cong G\oplus G'$ so that the direct sum
polarization on $G\oplus G'$ corresponds to the canonical polarization
of $J_d[p]$.

This completes the proof of part (3) of Theorem~\ref{thm:main}.
\qed\medskip

The following result establishes Theorem~\ref{thm:main}, parts (1) and
(2) for $p>3$.  Recall that $J_d$ is the Jacobian of the curve $\CC_d$
with affine equation $y^d=x(1-x)$.

\begin{thm}\label{thm:G-in-Cd}
  If $p>3$, then every $BT_1$ group scheme over $k$ appears as a
  direct factor of $J_{d}[p]$ for an integer $d$ of the form
  $d=p^\ell-1$.  The same holds for polarized $BT_1$ group schemes.
\end{thm}

\begin{proof}
  Suppose that $p>3$ and let $G$ be a $BT_1$ group scheme over $k$.
  Let $\{(\overline w_i)^{e_i}\}$ be the multiset of distinct
  primitive cyclic words corresponding to $G$ in the Kraft
  classification, and let $\ell_i$ be the length of $w_i^{e_i}$.  Let
  $d_i=p^{\ell_i}-1$ and let $S_i=\Z/d_i\Z\setminus\{0,d_i/2\}$ with
  the usual partition and permutation.  According to part (2) of
  Proposition~\ref{prop:mult>0}, there is an element $a\in S_i$ with
  $w_a=w_i^{e_i}$, and using Theorem~\ref{thm:BT1fermat}, we conclude
  that the group scheme $G_i$ with $M(G_i)\cong M(w_i^{e_i})$ appears
  as a direct factor of $J_{d_i}[p]$. 

  If $d'$ divides $d$, then there is a natural quotient morphism
  $\pi:\CC_d\to\CC_{d'}$ of degree $d/d'$, which is prime to $p$.  The
  induced composition
  $J_{d'}\labeledto{\pi^*} J_d\labeledto{\pi_*} J_{d'}$ is
  multiplication by $d/d'$ and therefore induces an isomorphism on
  $J_{d'}[p]$.  Thus $J_{d'}[p]$ is a direct factor of $J_d[p]$.

  Now let $\ell$ be the least common multiple of the $\ell_i$ (so that
  $d_i$ divides $d=p^\ell-1$ for all $i$).  Using the maps
  $J_{d_i}\to J_d$ shows that each $G_i$ is a direct factor of $J_d[p]$, and
  since the $G_i$ have pairwise non-isomorphic indecomposable factors,
  $G=\oplus G_i$ is a direct factor of $J_d[p]$ as well.  This completes the
  proof for $BT_1$ group schemes without polarization.

  The polarized case follows from the unpolarized case and
  Proposition~\ref{prop:self-dual=>polarized} by the same argument
  given at the end of the proof of part (3).  This completes the proof
  of the theorem.
\end{proof}

The following result establishes Theorem~\ref{thm:main} parts (1) and
(2) for $p>2$, and it handles the main case for $p=2$.  Recall that
$J_{F_d}$ is the Jacobian of the Fermat curve with affine equation
$X^d+Y^d=1$.

\begin{thm}\label{thm:G-in-JFd}\mbox{}
  \begin{enumerate}
  \item If $p>2$, then every $BT_1$ group scheme over $k$ appears as a
    direct factor of $J_{F_d}[p]$ for an integer $d$ of the form
    $d=p^\ell-1$.  The same holds for polarized $BT_1$ group schemes.
    \item The same is true if $p=2$ as long as the group scheme has no 
    factors of $\Z/2\Z$ or $\mu_2$. 
  \end{enumerate}
\end{thm}

\begin{proof}
  We omit many details, because the argument is quite parallel to the
  proof of Theorem~\ref{thm:G-in-Cd}.  Indeed, we may run the argument
  of Theorem~\ref{thm:G-in-Cd} where Theorem~\ref{thm:BT1Fd} plays the
  role of Theorem~\ref{thm:BT1fermat} and where
  Proposition~\ref{prop:mult>0} is replaced by an analogue where the
  set $T$ associated to the Fermat curve $F_d$ plays the role of the
  set $S$ associated to $\CC_d$.

  There is one special case that needs separate treatment, namely when
  $p=3$ and $G$ is $\Z/3\Z$ or $\mu_3$.  In this case, one checks
  using Theorem~\ref{thm:BT1Fd} that $G$ is a direct factor of
  $J_{F_8}[3]$ (and not of $J_{F_2}[3]$, since $F_2$ is rational).

  For the general case, let $w$ be a word of length $\ell$, and assume
  $p^\ell>3$.  If $p=2$, the hypothesis that $G$ has no factors of
  $\Z/2\Z$ or $\mu_2$ implies that $w\neq f^\ell$ and $w\neq v^\ell$.
  Let $d=p^\ell-1$, and consider the set $T$ with partition and
  permutation as in Section~\ref{ss:Fd}.  Write $w_{a,b}$ for the word
  associated to $(a,b)\in T$.  Then it will suffice to show that there
  is an element $(a,b)\in T$ with $w_{a,b} = w$.

  If $p>2$, $p^\ell>3$, and $w=f^\ell$ (resp.~$w=v^\ell$), we may take
  $(a,b)=(-1,-1)$ (resp.~$(a,b)=(1,1)$) and have $w_{a,b}=w$, so from
  now on we assume that $\ell>1$, $w\neq f^\ell$, and $w\neq v^\ell$.

  There is an injection $S\into T$ sending $a$ to
  $(a,a)$.  It is compatible with the partitions $S=S_f\cup S_v$ and
  $T=T_f\cup T_v$ and intertwines the permutations $\pi$ and
  $\sigma$.  If $a\in S$ has the property that
  $w_a=w$, then this shows that $w_{a,a}=w$ as well.  Thus part (1) of
  Proposition~\ref{prop:mult>0} gives an $a\in S$ with $w_{a,a}=w$
  when $w$ is primitive.  We may now assume that $w$ is not primitive,
  say $w={w'}^e$ where $e>1$ and $w'$ is primitive of length
  $\lambda$.

  Let $a'$ be the element of $S$ associated to $w$ as in the proof of
  Proposition~\ref{prop:mult>0}, i.e., we use the digit $0$ for $v$ and
  the digit $p-1$ for $f$.
  Note that $w_{a'}=w_{a',a'}=w'$.  Since $w$ is a power of $w'$,
  consideration of digits shows that $a'\neq\pm1$ and that
  $a'\neq d/2\pm1$.  Thus $(a'+1,a'-1)$ is an element of $T$, and it
  is clear that $w_{a'+1,a'-1}=w$.  This establishes that every word
  of length $\ell$ (not equal to $f^\ell$ or $v^\ell$ if $p=2$) arises
  as $w_{a,b}$ for a suitable $(a,b)\in T$.  This completes the proof
  of Theorem~\ref{thm:G-in-JFd}.
\end{proof}

\subsection{The case $p=2$}\label{ss:p=2}
To finish the proof of Theorem~\ref{thm:main},
it remains to treat the case where $p=2$ and $G$ is a $BT_1$ group
scheme over $k$ with factors of $\Z/2\Z$ or $\mu_2$.  Write
\begin{equation*}
G\cong (\Z/2\Z)^{f_1}\oplus(\mu_2)^{f_2}\oplus G',
\end{equation*}
where $G'$ is a $BT_1$ group scheme with no factors of $\Z/2\Z$ and no
factors of $\mu_2$.  We have already proven that $G'$ is a direct factor of
$J_{F_d}[2]$ for a suitable value of $d$ of the form $2^\ell-1$.  We
choose one such value of $d$.

Let $r$ be an odd positive integer and let $X_r$ be the smooth,
projective curve over $k$ defined by
\[X_r:\qquad \left(x^2-x\right)\left(z^r-1\right)=1.\]
One computes that $X_r$ has genus $r-1$, and by
\cite[Prop.~3.2]{Subrao75}, it is ordinary, i.e., 
\[J_{X_r}[2]\cong\left(\Z/2\Z\oplus\mu_2\right)^{r-1}.\]

Choose $r\ge\max\{f_1,f_2\}+1$ and odd, and let $F_1$ be the Fermat
curve of degree 1 (given by $X+Y=1$).  
Consider the degree $r$ projection $X_r \to F_1$ given by
\[(x,z)\mapsto (X=x,Y=1-x).\]
Define $C$ as the fiber
product of that projection and the degree $d^2$ projection
$F_d \to F_1$.  Since $d$ and $r$ are odd, $J_{X_r}[2]$ and
$J_{F_d}[2]$ are direct factors of $J_C[2]$.  They are disjoint since
$X_r$ is ordinary, whereas $F_d$ has $p$-rank zero.  Thus
\[G \subset (\Z/2\Z \oplus \mu_2)^{r-1}\oplus G'\]
is a direct factor of $J_C[2]$.  This completes the proof of the case
$p=2$ of Theorem~\ref{thm:main}.  \qed

\begin{rem} \label{ss:p=2-again} Another approach to adding factors of
  $\Z/2\Z$ and $\mu_2$ to $G'$ is to use an argument similar to
  \cite[Cor.~4.7]{LMPTpp18b}.  Using that $F_d$ has CM, so lies in the
  $\mu$-ordinary locus, one finds curves $C$ whose Newton polygon is
  that of $F_d$ with additional segments of slopes 0 and 1 of
  arbitrarily large multiplicity $f$.  Again using $\mu$-ordinarity,
  one deduces that $G'\oplus(\Z/2\Z\oplus\mu_2)^f$ is a direct factor
  of $J_C[2]$, thus so is $G$.  This method has the drawbacks that the
  curve $C$ is no longer explicit and we have no control over its
  field of definition other than that it is a finite field.
\end{rem}

\begin{rem} \label{rem:otherway}
  A weaker version of parts (1) and (2) of Theorem~\ref{thm:main}
  follows from the facts that each E--O stratum of $\AA_g$ is
  non-empty and that every abelian variety $A$ appears as a subvariety
  of a Jacobian $J$.  However, we need $A[p]$ to be a direct factor of
  $J[p]$.  If $A$ has dimension $\ell$, this can be verified when
  $p >3$ and $p \geq \ell$, via the theory of Prym--Tyurin varieties,
  \cite[Corollary~12.2.4]{birkenhakelange}.  As discussed in the
  introduction, our proof avoids this restriction on $p$ and gives
  more information about the curve.
\end{rem}

\bibliography{database}

\end{document}